\documentclass[12pt]{article}

\usepackage{amsfonts}
\usepackage{amssymb}
\usepackage{amsthm}
\usepackage{amsmath}
\usepackage{dsfont}
\usepackage{cite}
\usepackage{fullpage}

\DeclareMathOperator{\dist}{dist} \DeclareMathOperator{\D}{D}
\DeclareMathOperator{\HolSh}{HolSh}\DeclareMathOperator{\StSh}{StSh}
\DeclareMathOperator{\LipSh}{LipSh}
\DeclareMathOperator{\Diff}{Diff}\DeclareMathOperator{\SSS}{SS}
\DeclareMathOperator{\FinHolSh}{FinHolSh}
\DeclareMathOperator{\SG}{SG}
\DeclareMathOperator{\Id}{Id}

\newcommand{\R}{\ensuremath{\mathds{R}}}

\newcommand{\Z}{\ensuremath{\mathds{Z}}}

\newcommand{\cA}{\ensuremath{\mathcal{A}}}

\newcommand{\eps}{\varepsilon}

\newcommand{\al}{\alpha}
\newcommand{\be}{\beta}
\newcommand{\lam}{\lambda}

\newcommand{\ZZ}{\mathds{Z}}

\newcommand{\Mane}{\mbox{Ma$\tilde{\mbox{n}}\acute{\mbox{e}}$}}
\newcommand{\Holder}{\mbox{H\"older }}
\newcommand{\sref}[1]{(\ref{#1})}

\newtheorem{thm}{Theorem}[section]
\newtheorem{proposition}{Proposition}[section]
\newtheorem{lem}{Lemma}[section]
\newtheorem{quest}{Question}[section]
\newtheorem{conj}{Conjecture}[section]
\theoremstyle{definition}
\newtheorem{defin}{Definition}[section]
\newtheorem{example}{Example}[section]
\newtheorem{remark}{Remark}[section]

\begin{document}


\title{\Holder Shadowing on Finite Intervals}

\author{SERGEY TIKHOMIROV \footnote{Chebyshev Laboratory, Saint-Petersburg State Univeristy, 14th Line 29B, Vasilyevsky Island, St.Petersburg 199178, Russia. Max Planck Institute for Mathematics in the Science, Inselstrasse 22, 04103 Leipzig, Germany. Email: sergey.tikhomirov@gmail.com, Sergey.Tikhomirov@mis.mpg.de}
}

\maketitle

\begin{abstract}
For any $\theta, \omega > 1/2$ we prove that, if any $d$-pseudotrajectory of length $\sim 1/d^{\omega}$ of a diffeomorphism~$f\in C^2$ can be $d^{\theta}$-shadowed by an exact trajectory, then $f$ is structurally stable. Previously it was conjectured \cite{Yorke1, Yorke2} that for $\theta = \omega = 1/2$ this property holds for a wide class of non-uniformly hyperbolic diffeomorphisms. In the proof we introduce the notion of sublinear growth property for inhomogenious linear equations and prove that it implies exponential dichotomy.
\end{abstract}

\section{Introduction and Main Results}
The theory of shadowing of approximate trajectories
(pseudotrajectories) of dynamical systems is now a well-developed
part of the global theory of dynamical systems (see,
the monographs \cite{PilBook, PalmBook} and \cite{PilRev} for overview of modern results). The shadowing problem is related to the following question: under which conditions, for any pseudotrajectory of $f$ there exists a close trajectory?

The study of this problem was originated by Anosov \cite{Ano} and Bowen \cite{Bow}. This theory is closely related to the classical theory of structural stability. It is well known that a diffeomorphism has shadowing
property in a neighborhood of a hyperbolic set \cite{Ano, Bow} and a
structurally stable diffeomorphism has shadowing property on the
whole manifold \cite{Rob, Mor, Saw}.

Let $M$ be a smooth compact manifold of class $C^{\infty}$ without
boundary with Riemannian metric $\dist$. Consider a
diffeomorphism $f \in \Diff^1(M)$.
For an interval $I = (a, b)$, where $a \in \ZZ \cup \{-\infty\}$, $b \in \ZZ \cup \{+\infty\}$ and $d>0$ a sequence of points $\{y_k\}_{k \in I}$ is called a \textit{$d$-pseudotrajectory} if the following inequalities holds
$$
\dist(y_{k+1}, f(y_k)) < d, \quad k \in \ZZ, \quad k, k+1 \in I.
$$

\begin{defin}\label{defStSh}
We say that $f$ has the \textit{standard shadowing property}
($\StSh$) if for any $\eps > 0$ there exists $d>0$ such that for any
$d$-pseudotrajectory $\{y_k\}_{k\in \ZZ}$ there exists a trajectory
$\{x_k\}_{k\in \ZZ}$ such that
\begin{equation}\label{sh}
\dist(x_k, y_k) < \eps, \quad k \in \ZZ.
\end{equation}
In this case we say that pseudotrajectory
$\{y_k\}$ is \textit{$\eps$-shadowed} by~$\{x_k\}$.
\end{defin}

Currently pseudotrajectories of finite length are almost not investigated. This problem is strongly related to the dependence between $\eps$ and $d$ in the Definition \ref{defStSh}. In the present paper we study shadowing properties on finite intervals with polynomial dependence of $\eps$ and $d$.
We give an upper bound for the length of shadowable pseudotrajectories for non-hyperbolic systems.

The most well-know definition relating $\eps$ and $d$ is the following.
\begin{defin}\label{defLipSh}
We say that $f$ has the \textit{Lipschitz shadowing property} ($\LipSh$) if
there exist constants $d_0, L > 0$ such that for any $d < d_0$ and
$d$-pseudotrajectory $\{y_k\}_{k\in \ZZ}$ there exists a trajectory
$\{x_k\}_{k\in \ZZ}$ such that inequalities~\sref{sh} hold with $\eps = Ld$.
\end{defin}
Analyzing the proofs of the first shadowing results by Anosov~\cite{Ano} and Bowen~\cite{Bow}, it is easy to see that, in a neighborhood of a hyperbolic set, the shadowing property is Lipschitz (and the same holds in the case of a
structurally stable diffeomorphism~\cite{PilBook}). Recently \cite{PilTikhLipSh} it was proved
\begin{thm}\label{thmLip} A diffeomorphism $f \in C^1$ has Lipschitz shadowing property if and only if it is structurally stable.
\end{thm}
Let us also mention works~\cite{PilVar, LipPerSh}, where a similar result was proved for the case of periodic and variational shadowing properties.

At the same time, it is easy to give an example of a diffeomorphism
that is not structurally stable but has standard shadowing property (see
\cite{PilVar}, for instance).

Let us introduce the following definition, which is the central notion under investigation in the present article.
\begin{defin}\label{defFinHolSh}
We say that $f$ has the \textit{Finite \Holder shadowing property} with
exponents \hbox{$\theta \in (0, 1)$}, $\omega \geq 0$ ($\FinHolSh(\theta, \omega)$) if there exist constants $d_0, L, C > 0$ such that for any $d < d_0$ and
$d$-pseudotrajectory $\{y_k\}_{k \in [0, Cd^{-\omega}]}$ there exists a trajectory
$\{x_k\}_{k \in [0, Cd^{-\omega}]}$ such that
\begin{equation}\label{shHol}\notag
\dist(x_k, y_k) < L d^{\theta}, \quad k \in [0, Cd^{-\omega}].
\end{equation}
\end{defin}

The main result of the present paper is the following.

\begin{thm}\label{thmMain}
If a diffeomorphism $f \in C^2$ satisfies $\FinHolSh(\theta, \omega)$ with
\begin{equation}\label{Add3.1'}
\theta > 1/2, \quad \theta + \omega > 1
\end{equation}
then $f$ is structurally stable.
\end{thm}

Note that previously  S. Hammel, J. Yorke and C. Grebogi  based on results of numerical experiments conjectured the following \cite{Yorke1, Yorke2}:
\begin{conj}
A typical dissipative map $f: \R^2  \to \R^2$ satisfies $\FinHolSh(1/2, 1/2)$.
\end{conj}
This conjecture suggests us that Theorem \ref{thmMain} cannot be improved.

Theorem \ref{thmMain} has an interesting consequence even for the case of    infinite pseudotrajectories.
\begin{defin}\label{defHolSh}
We say that $f$ has \textit{\Holder shadowing property} with
exponent $\theta \in (0, 1)$ ($\HolSh(\theta)$) if there exist
constants $d_0, L > 0$ such that for any $d < d_0$ and
$d$-pseudotrajectory $\{y_k\}_{k \in \ZZ}$ there exists a trajectory
$\{x_k\}_{k \in \ZZ}$ such that inequalities~\sref{sh} hold with $\eps = Ld^{\theta}$.
\end{defin}

It is easy to see that for any $\theta \in (0, 1)$ and $\omega \geq 0$ the following inclusions hold
$$
\SSS = \LipSh \subset \HolSh(\theta) = \FinHolSh(\theta, +\infty) \subset \FinHolSh(\theta, \omega),
$$
where $\SSS$ denotes the set of structurally stable diffeomorphisms and $\LipSh$, $\HolSh$, $\FinHolSh$ denote sets of diffeomorphisms satisfying the corresponding shadowing properties.

The following theorem is a straightforward consequence of Theorem \ref{thmMain}.

\begin{thm}\label{thmHol}
If a diffeomorphism $f \in C^2$ satisfies $\HolSh(\theta)$ with $\theta > 1/2$ then $f$ is structurally stable.
\end{thm}
Note that this theorem generalizes Theorem \ref{thmLip}. Let us also mention a related work \cite{PujKor}, where some consequences of \Holder shadowing for 1-dimensional maps were proved.

It is worth to mention a relation between Theorem \ref{thmHol} and a question suggested by Katok:
\begin{quest}\label{q2}
Is every diffeomorphism that is \Holder conjugate to an Anosov diffeomorphism itself Anosov?
\end{quest}

Recently it was shown that in general the answer to
Question \ref{q2} is negative~\cite{Gog}. At the same time the following positive result was proved in \cite{Gog, Fisher}.
\begin{thm}\label{thmGog}
A $C^2$-diffeomorphism that is conjugate to an Anosov diffeomorphism
via \Holder conjugacy $h$ is Anosov itself, provided that the product
of \Holder exponents for $h$ and $h^{-1}$ is greater than $1/2$.
\end{thm}

It is easy to show that diffeomorphisms which are \Holder conjugate to a structurally stable one satisfy \Holder shadowing property. As a consequence of Theorem \ref{thmHol} we prove that a $C^2$-diffeomorphism that is conjugate to a structurally stable diffeomorphism via \Holder conjugacy $h$ is structurally stable itself, provided that the product of \Holder exponents for $h$ and $h^{-1}$ is greater than $1/2$, which generalizes Theorem~\ref{thmGog}.

To illustrate that Theorems \ref{thmMain}, \ref{thmHol} are almost sharp we give the following example.
\begin{example}\label{exMain} There exists a non-structurally stable
$C^{\infty}$-diffeomorphism $f: S^1 \to S^1$ satisfying $\HolSh(1/3)$ and $\FinHolSh(1/2, 1/2)$.
\end{example}
It is easy to see that the identity map satisfies $\FinHolSh(\theta, \omega)$ provided that $\theta + \omega \leq 1$.

The paper is organized as follows.

In section \ref{secSG} we introduce the main technical tools of the proof: the notion of slow growth property of inhomogenious linear equation and recall the notion of exponential dichotomy. We state Lemma \ref{lem1} and Theorem~\ref{SubLinThm} which are main steps of the proof and show that Theorem~\ref{thmMain} is their consequence.

In sections \ref{secHolLem}, \ref{secSGProof} we give the proofs Lemma \ref{lem1} and Theorem~\ref{SubLinThm}.

In section \ref{sec5} we describe Example \ref{exMain} and prove its properties.

\section{Slow Growth Property and Exponential Dichotomy}\label{secSG}

Consider Euclidian spaces $E_{n \in \Z}$ of dimension $m$ and a sequence $\cA = \{A_{n \in \Z} : E_n \to E_{n+1}\}$  of linear isomorphisms satisfying for some $R > 0$ the following inequalities
\begin{equation}\label{SG1.1}
\|A_n\|, \|A_n^{-1}\| < R, \quad n \in \Z.
\end{equation}

\begin{defin}
We say that a sequence $\cA$ has \textit{slow growth} property with exponent $\gamma > 0$ ($\cA \in \SG(\gamma)$) if there exists a constant $L > 0$ such that for any $i \in \Z$, $N > 0$ and a sequence $\{w_k \in E_k\}_{k \in [i+1, i+N]}$, $|w_k|\leq 1$ there exists a sequence $\{v_k \in E_k\}_{k \in [i, i+N]}$ satisfying
\begin{equation}\label{SG1.1.5}
v_{k+1} = A_k v_k + w_{k+1}, \quad k \in [i, i+N-1],
\end{equation}
\begin{equation}\label{SGAdd4.1}
|v_{k}| \leq LN^{\gamma}, \quad k \in [i, i+N].
\end{equation}
If $\cA \in \SG(\gamma)$ with $\gamma \in [0, 1)$ we say that it has \textit{sublinear growth} property. If $\cA \in \SG(0)$ we say that it has \textit{bounded solution} property.
\end{defin}

We have not found analogues of the notion of slow growth property in the literature. At the same time the notion of bounded solution property was widely investigated, for example see \cite{Maizel, Coppel, Pli, Palm1, Palm2, PalmNew, Lat, Bar}.

We prove the following relations between shadowing and sublinear growth properties.

\begin{lem}\label{lem1}
If $f$ satisfies assumptions of Theorem \ref{thmMain} then there exists $\gamma \in (0, 1)$ such that for any trajectory $\{p_k\}_{k \in \Z}$ the sequence $\{A_k = \D f(p_k)\}$ satisfies $\SG(\gamma)$.
\end{lem}

To characterize sequences satisfying sublinear growth property we need notion of exponential dichotomy (see \cite{Coppel}, for some generalisations see \cite{Bar}).

\begin{defin}\label{defED}
We say that a sequence $\cA$ has \textit{exponential dichotomy} on $\Z^+$ if there exist numbers $C > 0$, $\lam \in (0, 1)$ and a decomposition $E_k = E_k^{s, +} \oplus E_k^{u, +}$, $k \geq 0$ such that
\begin{equation}\notag
E_{k+1}^{\sigma, +} = A_k E_k^{\sigma, +}, \quad k \geq 0, \sigma \in \{s, u\},
\end{equation}
\begin{equation}\label{SG1.2}
|A_{k+l-1}\cdot \dots A_k v_k^s| \leq C \lam^l|v_k^s|, \quad k \geq 0, l > 0, v_k^s \in E_k^{s, +},
\end{equation}
\begin{equation}\label{SG1.3}
|A_{k+l-1}\cdot \dots A_k v_k^u| \geq \frac{1}{C} \lam^{-l}|v_k^u|, \quad k \geq 0, l > 0, v_k^u \in E_k^{u, +}.
\end{equation}
Similarly we say that $\cA$ has exponential dichotomy on $\Z^-$ if
there exist numbers $C > 0$, $\lam \in (0, 1)$ and a decomposition $E_k = E_k^{s, -} \oplus E_k^{u, -}$, $k \leq 0$ such that
\begin{equation}\notag
E_{k+1}^{\sigma, -} = A_k E_k^{\sigma, -}, \quad k < 0, \sigma \in \{s, u\},
\end{equation}
\begin{equation}\notag
|A_{k+l-1}\cdot \dots A_k v_k^s| \leq C \lam^l|v_k^s|, \quad l >0, l + k < 0, v_k^s \in E_k^{s, -},
\end{equation}
\begin{equation}\notag
|A_{k+l-1}\cdot \dots A_k v_k^u| \geq \frac{1}{C} \lam^{-l}|v_k^u|, \quad l > 0, l+k <0, v_k^u \in E_k^{u, -}.
\end{equation}
\end{defin}
Denote by $P_k^{s, +}$ the projection with the range $E_k^{s, +}$ and kernel $E_k^{u, +}$. Similarly we define $P_k^{u, +}$, $P_k^{s, -}$, $P_k^{u, -}$.

\begin{remark}\label{remH}
It is easy to show that there exists $H > 0$ such that (see for instance \cite[Remark 2.3]{Todorov})
$$
|P_k^{\sigma, a} v_k| \leq H|v_k|, \quad v_k \in E_k, \sigma \in \{s, u\}, a \in \{+, -\}, k \in \Z^a.
$$
\end{remark}
\begin{remark}
In Definition \ref{defED} we do not require the uniqueness of $E_k^{s, +}$, $E_k^{s, -}$, $E_k^{u, +}$, $E_k^{u, -}$. At the same time if $\cA$ has exponential dichotomy on $\Z^+$ then $E_k^{s, +}$ is uniquely defined and if $\cA$ has exponential dichotomy on $\Z^-$ then $E_k^{u, -}$ is uniquely defined \cite[Proposition 2.3]{PalmNew}.
\end{remark}

Recently the following were shown \cite[Theorem 1, 2]{Todorov}:
\begin{thm}\label{thmPliss}
A sequence $\cA$ has bounded solution property if and only if the following two conditions hold:
\begin{itemize}
\item[(ED)] $\cA$ has exponential dichotomy both on $\Z^+$ and $\Z^-$.
\item[(TC)] The corresponding spaces $E_0^{s, +}$, $E_0^{u, -}$ satisfy the following transversality condition
    $$
    E_0^{s, +} + E_0^{u, -} = E_0.
    $$
\end{itemize}
\end{thm}
\begin{thm}\label{thmCoppel}
The following statements are equivalent.
\begin{itemize}
  \item[(i)] $\cA$ has exponential dichotomy on $\Z^+$ ($\Z^-$).
  \item[(ii)] There exists $L > 0$ such that for any sequence $\{w_k \in E_k\}$, $k \geq 0$ ($k \leq 0$), satisfying $|w_k| \leq 1$ there exists sequence $\{v_k \in E_k\}_{k \in \Z}$ such that $|v_k| \leq L$ and
      \begin{equation}\label{SGAdd2.1}
      v_{k+1} = A_k v_k + w_{k+1}
      \end{equation}
      for $k \geq 0$ ($k \leq 0$).
\end{itemize}
\end{thm}

\begin{remark}\label{rem2}
Such type of results were also considered in \cite{Coff, PalmNew, Vietnam, Slyus}, however we were not able to find in earlier literature statements which follow Theorems~\ref{thmPliss},~\ref{thmCoppel}. Similar results not for sequences of isomorphisms but for inhomogeneous linear systems of differential equations was obtained in \cite{Coppel, Maizel,PalmBook,Pli}. The relation between discrete and continuous settings is discussed in
\cite{PilGen}.
\end{remark}

In this paper we prove the following theorem, which is interesting by itself without relation to shadowing property.
\begin{thm}\label{SubLinThm}
If a sequence $\cA$ has sublinear growth property then it satisfies properties (ED) and (TC).
\end{thm}
As a consequence of this theorem we conclude that sublinear growth property and bounded solution property are in fact equivalent.

\begin{remark}
Note that sequences $\cA \in\SG(1)$ do not necessarily satisfy condition (ED). A trivial example in arbitrary dimension is $\cA = \{A_k = \Id\}$.
\end{remark}

Let us now give the proof of Theorem~\ref{thmMain}.

For a point $p\in M$ we define the following two subspaces of $T_pM$:
$$
B^+(p)=\{v\in T_pM\,:\,|Df^k(p)v|\to 0,\quad k\to +\infty\},
$$
$$
B^-(p)=\{v\in T_pM\,:\,|Df^k(p)v|\to 0,\quad k\to -\infty\}.
$$

\begin{thm}[\Mane, \cite{Mane2}]\label{thmMane} Diffeomorphism $f$ is structurally stable
if and only if
$$
B^+(p)+B^-(p)=T_pM, \quad p \in M.
$$
\end{thm}

Theorem \ref{thmMain} is a consequence of Lemma \ref{lem1} and Theorems \ref{SubLinThm}, \ref{thmMane}.

\section{Proof of Lemma \ref{lem1}}\label{secHolLem}
Let $\exp$ be the standard exponential mapping on the tangent bundle
of $M$ and let $\exp_x: T_xM \to M$ be the corresponding exponential
mapping at a point $x$.
Denote by $B(r,x)$ the ball in $M$ of radius $r$ centered at a point
$x$ and by $B_T(r,x)$ the ball in $T_xM$ of radius $r$ centered at
the origin.

There exists $\eps>0$ such that, for any $x\in M$, $\exp_x$ is a
diffeomorphism of $B_T(\eps,x)$ onto its image, and $\exp^{-1}_x$ is
a diffeomorphism of $B(\eps,x)$ onto its image. In addition, we may
assume that $\eps$ has the following property.

If $v,w\in B_T(\eps,x)$, then
\begin{equation}
\label{eq1} \frac{\mbox{dist}(\exp_x(v),\exp_x(w))}{|v-w|}\leq 2;
\end{equation}
if $y,z\in B(\eps,x)$, then
\begin{equation}
\label{eq2}
\frac{|\exp^{-1}_x(y)-\exp^{-1}_x(z)|}{\mbox{dist}(y,z)}\leq 2.
\end{equation}

Let $L, C, d_0 > 0$ and $\theta \in (1/2, 1)$, $\omega > 0$ be the constants from the definition of $\FinHolSh$. Denote $\al = \theta - 1/2$. Inequalities \sref{Add3.1'} imply that
\begin{equation}\label{Add7.1}
\al \in (0, 1/2), \quad 1/2 - \al < \omega.
\end{equation}

Since $M$ is compact and $f \in C^2$ there exists $S > 0$ such that
\begin{equation}\label{H0}
\dist(f(\exp_x(v)), \exp_{f(x)}(\D f(x) v)) \leq S |v|^2, \quad
\mbox{$x \in M$, $v \in T_xM$, $|v|< \eps$},
\end{equation}
(we additionally decrease $\eps$, if necessarily).

Fix $i \in \Z$ and $N> 0$. For an arbitrary sequence $\{w_k \in T_{p_k}M\}_{k \in [i+1, i+N+1]}$ with $|w_k| \leq
1$ consider the following equations
\begin{equation}\label{H1}
v_{k+1} = A_k v_k + w_{k+1}, \quad k \in [i, i+N].
\end{equation}

For any sequence $\{v_k \in T_{p_k}M\}_{k \in [i, i+N+1]}$ denote $\|\{v_k\}\| = \max_{k \in [i, i+N+1]} |v_k|$. For any
sequence $\{w_k \in T_{p_k}M\}_{k \in [i+1, i+N+1]}$ consider the set
$$
E(i, N, \{w_k\}) = \left\{ \mbox{$\{v_k\}_{k \in [i, i+N+1]}$ satisfies
\sref{H1}} \right\}.
$$
Denote
\begin{equation}\label{Add3.1}
F(i, N, \{w_k\}) = \min_{\{v_k\}\in E(i, N, \{w_k\})}\|\{v_k\}\|.
\end{equation}

Since $\|\cdot\| \geq 0$ is a continuous function on the linear space of
sequences $\{v_k\}$ and the set $E(i, N, \{w_k\})$ is closed it follows that the value $F(i, N,
\{w_k\})$ is well-defined. Note that a sequence $\{v_k\} \in E(i, N,
\{w_k\})$ is determined by the value $v_{i}$. Consider the sequence
$\{v_k\}$ corresponding to $v_{i} = 0$. It is easy to see that
$|v_{i+k}| \leq 1+R+R^2 + \dots + R^{k}$ for $k \in [0, N+1]$,
where $R = \max_{x \in M} \|\D f(x)\|$. Hence $F(i, N, \{w_k\}) \leq
1+R+R^2 + \dots + R^{2N}$ for any $\{|w_k|\leq 1\}$. It is easy to
see that $F(i, N, \{w_k\})$ is continuous with respect to $\{w_k\}$ and
hence
\begin{equation}\label{Add3.2}
Q = Q(i, N) = \max_{\{w_k\}, \; |w_k|\leq 1}F(i, N, \{w_k\})
\end{equation}
is well defined.

Let us choose sequences $\{w_k\}$ and $\{v_k\} \in F(i, N, \{w_k\})$
such that
$$
Q(i, N) = F(i, N, \{w_k\}), \quad F(i, N, \{w_k\}) = \|\{v_k\}\|.
$$

The definition of $Q$ and linearity of equation \sref{H1} imply the following two properties.
\begin{itemize}
\item[(Q1)]
For any sequence $\{w_k'\}_{k \in [i+1, i+N+1]}$ there exists a sequence
$\{v_k'\}_{k \in [i, i+N+1]}$ satisfying
\begin{equation}\notag
v'_{k+1} = A_k v'_k + w'_{k+1}, \quad \|\{v_k'\}\| \leq Q(i, N) \|\{w_k'\}\|.
\end{equation}

\item[(Q2)] For any sequence $\{v_k\}_{k \in [i, i+N+1]}$, satisfying \sref{H1} holds the following inequality
$$
\|\{v_k\}\| \geq Q(i, N).
$$
\end{itemize}

Relations \sref{Add7.1} imply that there exists $\be > 0$ such that the following conditions holds
\begin{equation}\label{Add9}
0<(2+\be)(1/2-\al)< 1, \quad (2+\be)\omega > 1.
\end{equation}

Denote
$$
\gamma = \frac{1}{(2+ \beta)\omega} \in (0, 1), \quad \gamma' = 1 - (2 + \be)(1/2-\al) > 0,
$$
\begin{equation}\label{H1.5}
d = \frac{\eps}{Q^{2+\beta}}.
\end{equation}
Let us prove that there exist $L' > 0$ independent of $i$ and $N$ such that
\begin{equation}\label{Add8.1}
Q(i, N) \leq L' N^{\gamma}.
\end{equation}

Below we consider two cases.

\textit{Case 1.} $C((S+2)d)^{-\omega} < N$. Then
$Q < (\eps^{\omega}(S+2)^{\omega}/C)^{\gamma}N^{\gamma}$
and inequality \sref{Add8.1} is proved.

\textit{Case 2.} $C((S+2)d)^{-\omega} \geq N$. Below we prove even a stronger statement: there exists $L' > 0$ (independent of $i$ and $N$) such that
\begin{equation}\label{Add8.3}
Q(i, N) \leq L'.
\end{equation}

Considering the trajectory $\{p'_k = f^{-i}(p_k)\}$ we can assume without loss of generality that $i = 0$.

Consider the sequence
$$
y_k = \exp_{p_{k}}(d v_{k}), \quad k \in [0, N].
$$
Let us show that $\{y_k\}$ is an $(S+2)d$-pseudotrajectory.
For $k \in [0, N]$ equations \sref{eq1}, \sref{H0} and inequalities
$|dv_k| < \eps$, $(dQ)^2 < d$ imply the following:
\begin{multline}
\dist(f(y_k), y_{k+1}) = \dist(f(\exp_{p_k}(d v_k)), \exp_{p_{k+1}}(d(A_k v_k + w_{k+1})) \leq \\
\leq \dist(f(\exp_{p_k}(d v_k)), \exp_{p_{k+1}}(d A_k v_k)) + \\
\dist(\exp_{p_{k+1}}(dA_k v_k), \exp_{p_{k+1}}(d(A_k v_k + w_k)))
\leq
\\
\leq S |d v_k|^2 + 2d \leq (S+2)d.
\end{multline}

We may assume that
\begin{equation}\label{Add3}
Q > \left((S+2)\eps/d_0\right)^{1/(2+\be)}.
\end{equation}
Indeed, the righthand side of \sref{Add3} does not depend on $N$, and if
$Q$ is smaller than the right side of~\sref{Add3} then we have already
proved \sref{Add8.1}. In the text below we make similar remarks several times to ensure that $Q$ is large enough.

Inequality \sref{Add3} implies that $(S+2)d < d_0$. Since $f \in
\FinHolSh(1/2 + \al, \omega)$ and the assumption of case 2 holds it follows that the pseudotrajectory $\{y_k\}_{k \in [0, N]}$ can be
$L((S+2)d)^{1/2+\al}$-shadowed by a
trajectory $\{x_k\}_{k \in [0, N]}$.

By reasons similar to \sref{Add3} we may assume that
$L((S+2)d)^{1/2+\al} < \eps/2$. Inequalities \sref{eq1} and
\sref{Add3} imply that for $k \in [0, N]$ the following inequalities hold
\begin{multline*}
\dist(p_k, x_k) \leq \dist(p_k, y_k) + \dist(y_k, x_k) \leq
2d|v_k| + L((S+2)d)^{1/2+\al} < \eps.
\end{multline*}
Hence $c_k = \exp_{p_k}^{-1}(x_k)$ is well-defined.

Denote $L_1 = L(S+2)^{1/2 + \al}$. Since $\dist(y_k, x_k) < L_1
d^{1/2 + \al}$, inequalities \sref{eq2} imply that
\begin{equation}\label{H2.1}
|dv_k - c_k| < 2L_1 d^{1/2 + \al}.
\end{equation}
Hence
\begin{equation}\label{Add8}
|c_k| < Q d + 2L_1 d^{1/2 + \al}.
\end{equation}
By the reasons similar to \sref{Add3} we can assume that $|c_k|<
\eps$.

Since $f(x_k) = x_{k+1}$ inequalities \sref{eq2} and \sref{H0} imply
that for $k \in [0, N]$ the following relations hold
\begin{multline}\label{H2.2}
|c_{k+1} - A_k c_k| < 2 \dist(\exp_{p_{k+1}}(c_{k+1}),
\exp_{p_{k+1}}(A_kc_k)) = \\
= 2 \dist(f(\exp_{p_k}(c_k)), \exp_{p_{k+1}}(A_kc_k)) \leq
2S|c_k|^2.
\end{multline}
Inequalities \sref{Add9}, \sref{H1.5}, \sref{Add8} imply that $|c_k|
< L_2 Qd$ for some $L_2 > 0$ independent of $N$.

Let $t_{k+1} = c_{k+1} - A_kc_k$. By inequality \sref{H2.2} it follows
that
$$
|t_k| \leq 2S|c_k|^2 \leq L_3 (Qd)^2
$$
for some $L_3 > 0$ independent of $N$. Property (Q1) implies that there exists a sequence $\{\tilde{c}_k
\in T_{p_k}M\}$ satisfying
$$
\tilde{c}_{k+1} - A_k \tilde{c}_k = t_{k+1}, \quad |\tilde{c}_k|
\leq QL_3(Qd)^2, \quad k \in [0, N].
$$
Consider the sequence $r_k = c_k -
\tilde{c}_k$. Obviously it satisfies the following conditions
\begin{equation}\label{Add2.1}
r_{k+1} = A_k r_k, \quad |r_k - c_k| \leq QL_3(Qd)^2, \quad k \in [0, N].
\end{equation}
Consider the sequence $e_k = \frac{1}{d}(dv_k - r_k)$. Equations
\sref{H2.1} and \sref{Add2.1} imply that
\begin{equation}\label{H3.1}
e_{k+1} = A_k e_k + w_k, \quad k \in [0, N]
\end{equation}
and
$$
|e_k| = \left|\frac{1}{d}\left( (dv_k - c_k) - (r_k - c_k)
\right)\right| \leq L_1 d^{-1/2 + \al} + L_3 Q^3 d, \quad k \in [0,
N].
$$
Property (Q2) implies that
$$
L_1 d^{-1/2 + \al} + L_3 Q^3 d \geq Q.
$$
By \sref{H1.5} the last inequality is equivalent to
$$
L_4Q^{-(2+\be)(-1/2 + \al)} + L_5 Q^{1-\be} \geq Q,
$$
where $L_4, L_5 > 0$ do not depend on $N$.
This inequality and \sref{Add9} imply that
$$
L_4 Q^{1-\gamma'} + L_5 Q^{1-\be} \geq Q.
$$
Hence
$$
L_4 Q^{1-\gamma'} \geq Q/2 \quad \mbox{or} \quad  L_5Q^{1-\be} \geq Q/2,
$$
and
$$
Q \leq \max((2L_4)^{1/\gamma'}, (2L_5)^{1/\be}).
$$
We have proved that there exists $L' > 0$ such that \sref{Add8.3} holds.
This completes the proof of Case 2 and Lemma
\ref{lem1}.

\section{Proof of Theorem \ref{SubLinThm}}\label{secSGProof}
Let us first prove the following.
\begin{lem}\label{lemSTC} If a sequence $\cA$ satisfies slow growth property and (ED) then it satisfies (TC).
\end{lem}
\begin{proof}
Let $L, \gamma > 0$ be the constants from the definition of slow growth property and let $C > 0$, $\lam \in (0, 1)$ be the constants from the definition of exponential dichotomy on $\Z^{\pm}$. Let $H$ be the constant from Remark \ref{remH} for exponential dichotomies on $\Z^{\pm}$. Assume that $E_0^{s, +} + E_0^{u, -} \ne E_0$.
Let us choose a vector $\eta \in E_0 \setminus (E_0^{s, +} + E_0^{u, -})$ satisfying $|\eta| = 1$. Denote $a = \dist(\eta, E_0^{s, +} + E_0^{u, -})$. Consider the sequence $\{w_k \in E_k\}_{k \in \Z}$ defined by the formula
$$
w_k =
\begin{cases}
0, &\quad k \ne 0,\\
\eta, &\quad k = 0.
\end{cases}
$$
Take $N > 0$ and an arbitrary solution $\{v_k\}_{k \in [-N, N]}$ of
\begin{equation}\label{SG3.0.5}
v_{k+1} = A_k v_k + w_k, \quad k \in [-N, N-1].
\end{equation}
Denote $v_k^{s, +} = P_k^{s, +}v_k$, $v_k^{u, +} = P_k^{u, +}v_k$ for $k \geq 0$. Since $w_k = 0$ for $k > 0$ we conclude
\begin{equation}\notag
|v_N^{u, +}| \geq \frac{1}{C}\lam^{-(N-1)}|v_0^{u, +}|
\end{equation}
and hence
\begin{equation}\label{SG3.1}
|v_N| \geq \frac{1}{H}\frac{1}{C} \lam^{-(N-1)}|v_0^{u, +}|.
\end{equation}
Similarly we denote
$v_k^{s, -} = P_k^{s, -}v_k$, $v_k^{u, -} = P_k^{u, -}v_k$, for $k \leq 0$
and conclude
\begin{equation}\label{SG3.2}
|v_{-N}| \geq \frac{1}{H}\frac{1}{C} \lam^{-(N-1)}|v_{-1}^{s, -}|.
\end{equation}
Equality \sref{SG3.0.5} implies that
$$
v_0 = A_{-1}v_{-1} + \eta
$$
and hence
$$
\max(\dist(v_0, E_0^{s, +} + E_0^{u, -}), \dist(A_{-1}v_{-1}, E_0^{s, +} + E_0^{u, -})) \geq a/2.
$$
From this inequality it is easy to conclude that
\begin{equation}\label{SG4.1}
v_0^{u, +} \geq \frac{1}{H}\frac{a}{2} \quad \mbox{or} \quad v_{-1}^{s, -} \geq \frac{1}{R}\frac{1}{H}\frac{a}{2}.
\end{equation}
Inequalities \sref{SG3.1}-\sref{SG4.1} imply that
$$
\max(|v_N|, |v_{-N}|) \geq \frac{1}{H}\frac{1}{C}\lam^{-(N-1)}\frac{1}{R}\frac{1}{H}\frac{a}{2}.
$$
Note that for large enough $N$ the right hand side of this inequality is greater than $L(2N+1)^{\gamma}$ which contradicts to the sublinear growth property.
\end{proof}

Now let us pass to the proof of Theorem \ref{SubLinThm}.
We prove this statement by induction over~$m$ (dimension of the Euclidian spaces).
First we prove the following.
\begin{lem}\label{lemSG7.5}
Theorem \ref{SubLinThm} holds for $m =1$.
\end{lem}
\begin{proof}
Choose a vector $e_0 \in E_0$, $|e_0| = 1$ and consider the sequence $\{e_k \in E_k\}_{k \in \Z}$ defined by the relations
\begin{equation}\label{SG5.1}
e_{k+1} = \frac{A_k e_k}{|A_k e_k|}, \; e_{-k-1} = \frac{A_{-k-1}^{-1} e_{-k}}{|A_{-k-1}^{-1} e_{-k}|} \quad k \geq 0.
\end{equation}
Let $\lam_k = |A_k e_k|$. Inequalities \sref{SG1.1} imply that
\begin{equation}\label{SGAdd11.1}
\lam_k \in (1/R, R), \quad k \in \Z.
\end{equation}

Denote
\begin{equation}\label{SG5.2}
\Pi(k, l) = \lam_k \cdot \dots \cdot \lam_{k+l-1}, \quad k \in \Z, l \geq 1.
\end{equation}

Let us prove the following lemma, which is the heart of the proof of Theorem \ref{SubLinThm}.
\begin{lem}\label{lemSG8}
If $m = 1$ and $\cA$ satisfies sublinear growth property then there exists $N> 0$ such that for any $i \in \Z$
$$
\Pi(i, N) > 2 \quad \mbox{or} \quad \Pi(i+N, N) < 1/2.
$$
\end{lem}
\begin{proof}
The proof follows the ideas of \cite{LipPerSh}.

Let us fix $i \in \Z$, $N > 0$ and consider the sequence
$$
w_k = -e_k, \quad k \in [i+1, i+2N+1].
$$
By sublinear growth property there exists a sequence $\{v_k\}_{k \in [i, i+2N+1]}$ satisfying
$$
v_{k+1} = A_k v_k + w_k,  \quad |v_k| \leq L (2N+1)^{\gamma}, \quad k \in [i, i+2N].
$$
Let $v_k = a_k e_k$, where $a_k \in \R$, then
\begin{equation}\label{SG5.5}
a_{k+1} = \lam_k a_k - 1, \quad |a_k| \leq L(2N+1)^{\gamma}, \quad k \in [i, i+2N].
\end{equation}
Those relations easily imply the following
\begin{proposition}\label{statAdd11}
If $a_k \leq 0$ for some $k \in [i, i+2N-1]$ then $a_{k+1} < 0$.
\end{proposition}

Below we prove the following: There exists a large $N > 0$ (depending only on $R$, $L$, $\gamma$) such that
\begin{itemize}
  \item[Case 1.] if $a_{i+N-1} \geq 0$ then $\Pi(i, N) > 2$,
  \item[Case 2.] if $a_{i+N-1} < 0$ then $\Pi(i+N, N) < 1/2$.
\end{itemize}
We give the proof of the case 1 in details, the second case is similar.
Proposition \ref{statAdd11} implies that $a_i, \dots, a_{i + N - 2} > 0, \; a_{i+N-1} \geq 0.$ Relation \sref{SG5.5} implies that
$$
\lam_k = \frac{a_{k+1}+1}{a_k}, \quad k \in [i, i+N-1].
$$
The following relations hold
\begin{multline*}
\Pi(i, N) = \frac{a_{i+1}+1}{a_i}\frac{a_{i+2}+1}{a_{i+1}} \dots  \frac{a_{i+N-1} + 1}{a_{i+N-2}} = \\
= \frac{1}{a_i} \frac{a_{i+1}+1}{a_{i+1}} \frac{a_{i+2}+1}{a_{i+2}}  \dots \frac{a_{i+N-2}+1}{a_{i+N-2}}(a_{i+N-1}+1) = \\
=\frac{a_{i+N-1}+1}{a_i} \prod_{k = i+1}^{i+N-2} \frac{a_k +1}{a_k} \geq \frac{1}{L(2N+1)^{\gamma}}\left( 1+ \frac{1}{L(2N+1)^{\gamma}} \right)^{N-2}.
\end{multline*}
Denote the latter expression by $G_{\gamma}(N)$. The inclusion $\gamma \in (0, 1)$ implies that
\begin{equation}\label{SGLim}
\lim_{N \to +\infty} G_{\gamma}(N) = +\infty
\end{equation}
and for large enough $N$ the inequality $G_{\gamma}(N) > 2$ holds, which completes the proof of Case~1.
\begin{remark}
In relation \sref{SGLim} we essentially use that $\gamma \in (0, 1)$; for $\gamma \geq 1$ it does not hold.
\end{remark}
\end{proof}
\begin{lem}\label{lemSG9}
Let $N$ be the number from Lemma \ref{lemSG8}.
\begin{itemize}
  \item[(i)] If $\Pi(i, N) > 2$ then $\Pi(i-N, N) > 2$.
  \item[(ii)] If $\Pi(i, N) < 1/2$ then $\Pi(i+N, N) < 1/2$.
\end{itemize}
\end{lem}
\begin{proof}
We prove statement (i); the second one is similar. Lemma \ref{lemSG8} implies that either $\Pi(i-N, N) > 2$ or $\Pi(i, N) < 1/2$. By the assumptions of Lemma~\ref{lemSG9} the second case is not possible and hence $\Pi(i-N, N) > 2$.
\end{proof}
Now let us complete the proof of Lemma \ref{lemSG7.5}. It is easy to conclude from Lemmas \ref{lemSG8}, \ref{lemSG9} that one of the following cases holds.
\begin{itemize}
  \item[Case 1.] For all $i \in \Z$ the inequality $\Pi(i, N) > 2$ holds. Then
      $$
      \Pi(i, l) \geq R^{N-1}(2^{1/N})^l, \quad i \in \Z, l > 0
      $$
      and hence $\cA$ has exponential dichotomy on $\Z^{\pm}$ with the splitting
      $$
      E_k^{s, \pm} = \{0\}, \quad E_k^{u, \pm} = \left<e_k\right>, \quad k \in \Z.
      $$
  \item[Case 2.] For all $i \in \Z$ the inequality $\Pi(i, N) < 1/2$ holds. Similarly to the previous case $\cA$ has exponential dichotomy on $\Z^{\pm}$ with the splitting
      $$
      E_k^{s, \pm} = \left<e_k\right>, \quad E_k^{u, \pm} = \{0\}.
      $$
  \item[Case 3.] There exist $i_1, i_2 \in \Z$ such that
  $$
  \Pi(i_1, N) > 2, \quad \Pi(i_2, N) < 1/2.
  $$
  Similarly to Case 1 the following inequality holds
  $$
  \Pi(k, l) \geq R^{N-1}(2^{1/N})^l, \quad k+l < i_1, l > 0
  $$
  and hence
  $$
  \Pi(k, l) \geq R^{|i_1| + N-1}(2^{1/N})^l, \quad k+l < 0, l > 0.
  $$
  The last inequality implies that $\cA$ has exponential dichotomy on $\Z^-$ with the splitting
      $$
      E_k^{s, -} = \{0\}, \quad E_k^{u, -} = \left<e_k\right>, \quad k \leq 0.
      $$
  Similarly
  $\cA$ has exponential dichotomy on $\Z^+$ with the splitting
      $$
      E_k^{s, +} = \left<e_k\right>, \quad E_k^{u, +} = \{0\}, \quad k \geq 0.
      $$
\end{itemize}
In all of those cases Lemma \ref{lemSG7.5} is proved.
\end{proof}

Now let us  continue the proof of Theorem \ref{SubLinThm}. Assume that Theorem \ref{SubLinThm} is proved for $\dim E_k \leq m$. Below we prove it for $\dim E_k = m+1$.

Let us choose a unit vector $e_0 \in E_0$ and consider the vectors $\{e_k\}_{k \in \Z}$ defined by relations~\sref{SG5.1}. Denote $\lam_k = |A_k e_k|$. Similarly to Lemma \ref{lemSG7.5} inclusions \sref{SGAdd11.1} hold.
For $k \in \Z$ let $S_k$ be the orthogonal complement of $e_k$ in $E_k$ and let $Q_k$ be the orthogonal projection onto $S_k$. Note that $\dim S_k = m$. Consider the linear operators $B_k : S_k \to S_{k+1}$, $D_k: S_k \to \left<e_{k+1}\right>$ defined by the following
$$
B_k = Q_{k+1}A_k, \quad D_k = (\Id - Q_{k+1})A_k, \quad k \in \Z.
$$
Note that $B_k^{-1} = Q_{k-1}A_k^{-1}$ and
\begin{equation}\label{SGAdd7.1}
\|B_k\|, \|B_k^{-1}\|, \|D_k\| < R.
\end{equation}
For any vector $b \in E_k$ denote by $b^{\perp} = P_k b, \quad b^1 = b - b^{\perp}.$ We also write $b = (b^{\perp}, b^1)$.
In such notation equations \sref{SG1.1.5} are equivalent to
\begin{equation}\label{SG9.2}
v_{k+1}^{\perp} = B_k v_k^{\perp} + w_{k+1}^{\perp},
\end{equation}
\begin{equation}\label{SG9.3}
v_{k+1}^1 = \lam_k v_k^1 + D_k v_k^{\perp} + w_{k+1}^{1}.
\end{equation}

Let us prove that the sequence $\{B_k\}$ satisfies property $\SG(\gamma)$.
Indeed, fix $i \in \Z$, $N > 0$ and consider an arbitrary sequence $\{w_k^{\perp} \in S_k\}_{k \in [i+1, i+N+1]}$ with $|w_k^{\perp}| \leq 1$. Consider the sequence $\{w_k \in E_k\}_{k \in [i+1, i+N+1]}$ defined by $w_k = w_k^{\perp}$. By the sublinear growth property there exists a sequence $\{v_k \in E_k\}_{k \in [i, i+N+1]}$ satisfying \sref{SG1.1.5}, \sref{SGAdd4.1} and hence \sref{SG9.2}. Recalling that $|v_k^{\perp}|\leq |v_k|$ we conclude that the sequence $\{B_k\}$ satisfies sublinear growth property and hence by the induction assumption if satisfies conditions (ED) and (TC) from Theorem~\ref{thmPliss}.

Below we prove that $\cA$ has exponential dichotomy on $\Z^+$. Let $\{B_i\}$ satisfy exponential dichotomy on $\Z^+$ with constants $C > 0$, $\lam \in (0, 1)$ and splitting $S_k = S_k^{s, +} \oplus S_k^{u, +}$. Let $H_1$ be the constant from Remark \ref{remH} for this splitting.

First we prove that there exists a big $N > 0$ such that for any $i \geq 2N$ the following inequality hold
\begin{equation}\label{SG29.0.7}
\Pi(i, N) > 2 \quad \mbox{or} \quad \Pi(i-N, N) < 1/2,
\end{equation}
where $\Pi(k, l)$ is defined by \sref{SG5.2}.

Let us choose $N > 0$ satisfying
\begin{equation}\label{SG29.0.5}
C \lam^N H_1L(4N)^{\gamma} < 1/(4R).
\end{equation}
and consider some $i \geq 2N$. Define a sequence $\{w_k = -e_k\}_{k \in [i-2N, i+2N]}$. By slow growth property there exists a sequence $\{v_k = (v_k^{\perp}, v_k^1)\}_{k \in [i-2N, i+2N+1]}$ satisfying the following for $k \in [i-2N, i+2N]$:
\begin{equation}\label{SG29.1.9}
v_{k+1}^{\perp} = B_k v_{k}^{\perp}
\end{equation}
\begin{equation}\label{SG29.2}
v_{k+1}^1 = \lam_k v_k^1 + D_k v_k^{\perp} - 1,
\end{equation}
\begin{equation}\label{SG29.1}
|v_k| < L(4N)^{\gamma}.
\end{equation}
Represent $v_k^{\perp} = v_k^{\perp, s} + v_k^{\perp, u}$, where
$v_k^{\perp, s} \in S_k^{s, +}$, $v_k^{\perp, u} \in S_k^{u, +}$.
Applying relations \sref{SG29.1.9}, \sref{SG29.1} and Remark \ref{remH} we conclude that
$$
|v_k^{\perp, s}|, |v_k^{\perp, u}| < H_1L(4N)^{\gamma}, k \in [i - 2N, i + 2N].
$$
Exponential dichotomy of $\{B_i\}$ implies that
$$
|v_k^{\perp, s}|, |v_k^{\perp, u}| < C \lam^N H_1L(4N)^{\gamma}, \quad k \in [i-N, i+N].
$$
By inequality \sref{SG29.0.5} we conclude that
$$
|v_k^{\perp, s}|, |v_k^{\perp, u}| < 1/(4R), \quad k \in [i-N, i+N]
$$
and hence
\begin{equation}\label{SG30.1}
|v_k^{\perp}| < 1/(2R), \quad k \in [i-N, i+N].
\end{equation}
Denote $b_k = D_kv_k^{\perp} - 1$. Inequalities \sref{SGAdd7.1} and \sref{SG30.1} imply that
$$
b_k \in (-3/2, -1/2), \quad k \in [i-N, i+N].
$$
Using those inclusions, relations \sref{SG29.2}, \sref{SG29.1} and arguing similarly to Lemma~\ref{lemSG8} (increasing $N$ if necessarily) we conclude relation \sref{SG29.0.7}.

Arguing similarly to the proof of Lemma \ref{lemSG7.5} we conclude that the linear operators generated by $\lam_i$ have exponential dichotomy on $\Z^+$.

Let us show that $\cA$ has exponential dichotomy on $\Z^+$. Consider an arbitrary sequence
$$
\{w_k = (w_k^{\perp}, w_k^1) \in E_k\}_{k \geq 0}, \quad |w_k| \leq 1.
$$
Since $\{B_k\}$ has exponential dichotomy on $\Z^+$, by Theorem \ref{thmCoppel} there exists a sequence $\{v_k^{\perp} \in S_k\}_{k \geq 0}$, satisfying \sref{SG9.2} and $|v_k| \leq L_1$, where $L_1 > 0$ does not depend on $\{w_k\}$. Inequality \sref{SGAdd7.1} implies that
$$
|D_k v_k^{\perp} + w_{k+1}^1| \leq L_1R + 1, \quad k \geq 0.
$$
Since linear operators generated by $\lam_k$ have exponential dichotomy on $\Z^+$, by Theorem \ref{thmCoppel} there exists $\{v_k^1 \in \R \}$ such that for $k \geq 0$ equalities \sref{SG9.3} hold and $|v_k^1| \leq L_2(L_1R + 1)$, where $L_2$ does not depend on $\{w_k\}$.

Hence for $k \geq 0$ the sequence $v_k = (v_k^{\perp}, v_k^1)$ satisfies \sref{SGAdd2.1} and
$$
|v_k| \leq |v_k^{\perp}| + |v_k^1| \leq L_2(L_1R + 1) + L_1.
$$
Theorem \ref{thmCoppel} implies that $\cA$ has exponential dichotomy on $\Z^+$.

Similarly $\cA$ has exponential dichotomy on $\Z^-$ and hence satisfies property (ED). By Lemma \ref{lemSTC} the sequence $\cA$ also satisfies property (TC). This completes the induction step and the proof of Theorem \ref{SubLinThm}.

\section{Example \ref{exMain}}\label{sec5}

Consider a diffeomorphism $f:S^1 \to S^1$ constructed as follows.
\begin{itemize}
\item[(i)] The nonwandering set of $f$ consists of two fixed points $s, u \in S^1$.
\item[(ii)] In some neighborhood $U_s$ of $s$ there exists a coordinate
system such that $f|_{U_s}(x) = x/2$.
\item[(iii)] In some neighborhood $U_u$ of $u$ there exists a coordinate
system such that $f|_{U_u}(x) = x + x^3$.
\item[(iv)] In $S^1 \setminus (U_s \cup U_u)$ the map is chosen to be $C^{\infty}$
and to satisfy the following condition: there exists $N> 2$ such that
$$
f^N(S^1 \setminus U_u) \subset U_s, \quad f^{-N}(S^1 \setminus U_s)
\subset U_u, \quad f^2(U_u) \cap U_s = \emptyset.
$$
\end{itemize}

\begin{thm}\label{thmEx}
If $f: S^1 \to S^1$ satisfies the above properties (i)--(iv) then $f \in
\HolSh(1/3)$ and $f \in \FinHolSh(1/2, 1/2)$.
\end{thm}

\begin{proof}
First let us prove a technical statement.
\begin{lem}\label{lemx3}
Denote $g(x) = x+x^3$. If $ |x-y| \geq \eps $ then
$$
|g(x) - g(y)| \geq \eps + \eps^3/4.
$$
\end{lem}

\begin{proof}
Using inequality $x^2+xy+y^2 > (x-y)^2/4$ we deduce that
\begin{multline*}
|g(x)- g(y)| = |x+x^3 - y - y^3|  = |(x-y)(1+x^2+xy+y^2)| \geq \\
\geq |(x-y)||1 + (x-y)^2/4| \geq \eps(1+\eps^2/4).
\end{multline*}
\end{proof}

We divide the proof of Theorem \ref{thmEx} into several propositions.

\begin{proposition}\label{propCase1}Conditions (ii), (iii) imply that there
exists $d_1 > 0$ such that
\begin{equation}\label{ex2}
B(d_1, f(U_s)) \subset U_s, \quad B(d_1, f^{-1}(U_u)) \subset U_u,
\quad B(d_1, f(S^1 \setminus U_u)) \subset S^1 \setminus U_u.
\end{equation}
Since $f|_{U_s}$ is hyperbolically contracting there exist $L > 0$ and $d_2 \in (0,d_1)$ such that for any
$d$-pseudotrajectory $\{y_k\}$ with $d< d_2$ and $y_0 \in S^1
\setminus U_u$ the following conditions hold
\begin{itemize}
\item $\{y_k\}_{k \geq 0} \subset S^1\setminus U_u$,

\item
$ \dist(f^k(x_0), y_k) < Ld,$ for $x_0 \in B(d, y_0), \; k \geq 0, $
\item if $\{y_k\}_{k \in \ZZ} \subset
S^1\setminus U_u$ then $\{y_k\}_{k \in \ZZ}$ can be $Ld$-shadowed by
a trajectory.
\end{itemize}
\end{proposition}

\begin{proposition}\label{propCase2} For any $d$-pseudotrajectory $\{y_k\}_{k \leq 0}$
with $d< d_1$ and $y_0 \in U_u$ the following inequality holds
\begin{equation}\label{ex4}
\dist(y_k,f^k(y_0))< 2 d^{1/3}, \quad k \leq 0.
\end{equation}
\end{proposition}
\begin{proof}
Proposition \ref{propCase1} implies that $y_k \in U_u$ for $k < 0$. Assume
\sref{ex4} does not hold. Let
$$
l = \max\{k\leq 0 : \dist(y_k,
f^k(y_0)) \geq 2 d^{1/3}\}.
$$
Note that $l<0$. Lemma \ref{lemx3}
implies that
$$
\dist(f(y_l), f^{l+1}(y_0)) > 2 d^{1/3} + 2d.
$$
Hence $\dist(y_{l+1}, f^{l+1}(y_0)) > 2 d^{1/3}$, which contradicts
to the choice of $l$.
\end{proof}

\begin{proposition}\label{propCase3}  If $\{y_k\}_{k \in \ZZ} \subset U_u$ is a
$d$-pseudotrajectory with $d< d_1$ then
\begin{equation}\label{ex5}
\dist(y_k, u) < 2d^{1/3}, \quad k \in \ZZ.
\end{equation}
\end{proposition}
\begin{proof}
Let us identify $y_k$ with its coordinate in the system
introduced in (iii) above and consider $Y = \sup_{k \in \ZZ}|y_k|$. Assume
that $Y > 2d^{1/3}$; then there exists $k \in \ZZ$ such that
\begin{equation}\label{ex3}\notag
|y_k| > \max(2d^{1/3}, Y-d/2).
\end{equation}
Without loss of generality we may assume that $y_k > 0$. Since $y_k
\in U_u$ the following holds
$$
f(y_k) - y_k = y_k^3 > 2d.
$$
Hence $y_{k+1} - y_k> (f(y_k) - y_k)-d > d$ and $y_{k+1} > Y+d/2$,
which contradicts to the choice of $Y$. Inequalities \sref{ex5} are
proved.
\end{proof}

\begin{proposition}\label{propCase4} For any $d$-pseudotrajectory $\{y_k\}_{k \in [0, n]}$ with $d < d_1$ and $y_n \in U_u$ the following inequality holds
\begin{equation}\label{Add17.1.1}
\dist(y_{n-k}, f^{-k}(y_n)) \leq dk, \quad k \in [0, n].
\end{equation}
\end{proposition}
\begin{proof}
Proposition \ref{propCase1} implies that $y_k \in U_u$ for $k \in [0, n]$. Assume that \sref{Add17.1.1} does not hold. Denote
$$
l = \min \{k \in [0, n] : \dist(y_{n-k}, f^{-k}(y_n)) > dk \}.
$$
Note that $l>0$. Lemma \ref{lemx3} implies that
$$
\dist(f(y_{n-l}), f^{-l+1}(y_n)) > ld
$$
and hence
$$
\dist(y_{n-l+1}, f^{-l+1}(y_n)) > (l-1)d,
$$
which contradicts to the choice of $l$.
\end{proof}

Now we are ready to complete the proof of Theorem \ref{thmEx}.

First let us prove that $f \in \HolSh(1/3)$.
Consider an arbitrary $d$-pseudotrajectory $\{y_k\}_{k \in \ZZ}$ with
$d < d_2$. Let us prove that it can be $L d^{1/3}$-shadowed by a trajectory.

If $\{y_k\} \subset U_u$ then by Proposition \ref{propCase3} it can be $2 d^{1/3}$-shadowed by $\{x_k = u\}$.

If $\{y_k\} \subset S^1 \setminus U_u$ then by Proposition \ref{propCase1} it can be $Ld$-shadowed.

In the other cases there exists $l$ such that $y_l \in U_u$ and $y_{l+1}
\notin U_u$. By Proposition \ref{propCase2}
$$
\dist(y_k, f^{k-l}(y_l)) < 2d^{1/3}, \quad k \leq l.
$$
By Proposition \ref{propCase1}
$$
\dist(y_k, f^{k-l}(y_l)) < Ld, \quad k \geq l+1.
$$
Hence $\{y_k\}$ is $Ld^{1/3}$-shadowed by the trajectory $\{x_k =
f^{k-l}(y_l)\}$.


Now let us prove that $f \in \FinHolSh(1/2, 1/2)$.
Consider an arbitrary $d$-pseudotrajectory $\{y_k\}_{k \in [0, 1/d^{1/2}]}$ with
$d < d_2$. Let us prove that it can be $L d^{1/2}$-shadowed by a trajectory.

If $\{y_k\} \subset U_u$ then by Proposition \ref{propCase4} it can be $d^{1/2}$-shadowed by $\{x_k = f^{k-n}(y_n)\}$.

If $\{y_k\} \subset S^1 \setminus U_u$ then by Proposition \ref{propCase1} it can be $L
d$-shadowed.

In the other cases there exists $l$ such that $y_l \in U_u$ and $y_{l+1}
\notin U_u$. From Proposition \ref{propCase4} it is easy to conclude that
$$
\dist(y_k, f^{k-l}(y_l)) < d^{1/2}, \quad k \leq l.
$$
Proposition \ref{propCase1} implies that
$$
\dist(y_k, f^{k-l}(y_l)) < Ld, \quad k \geq l+1.
$$
Hence $\{y_k\}$ is $Ld^{1/2}$-shadowed by the trajectory $\{x_k =
f^{k-l}(y_l)\}$.
\end{proof}

\section{Acknowledgement} The author would like to thank Anatole Katok
for introduction to problem of \Holder shadowing and Andrey Gogolev
for fruitful discussions. The work of the author was supported by the Chebyshev Laboratory (Department of Mathematics and Mechanics, St. Petersburg State University)  under RF Government grant 11.G34.31.0026, JSC ``Gazprom Neft'' and Humboldt Postdoctoral Fellowship (Germany).


\begin{thebibliography}{99}

\bibitem{Ano}
D. V. Anosov. On a class of invariant sets of smooth dynamical
systems. \textit{Proc. 5th Int. Conf. on Nonlin. Oscill.} {\bf 2}, Kiev,
1970, 39-45.

\bibitem{Bar}
L. Barreira, C. Valls. Stable manifolds for nonautonomous equations without exponential
dichotomy. \textit{J. Differential Equations} \textbf{221} (2006), 58–90.

\bibitem{Bow}
R. Bowen. Equilibrium States and the Ergodic Theory of Anosov
Diffeomorphisms. Lecture Notes Math., vol. 470, Springer, Berlin,
1975.

\bibitem{Lat}
C. Chicone; Yu. Latushkin. Evolution semigroups in dynamical systems and differential equations. Mathematical Surveys and Monographs, 70. American Mathematical Society, Providence, RI, 1999

\bibitem{Coff}
C. V. Coffman, J. J. Schaeffer.
Dichotomies for linear difference equations.
\textit{Math. Ann.} \textbf{172} (1967) 139–166.

\bibitem{Coppel}
W. A. Coppel. Dichotomies in stability theory. Lecture Notes in Mathematics, vol. 629, Berlin-Heidelberg-New York: Springer-Verlag, 1978.

\bibitem{Fisher}
T. Fisher, PhD Thesis. PennState, 2006.

\bibitem{Gog}
A. Gogolev. Diffeomorphisms \Holder conjugate to Anosov
diffeomorphisms. \textit{Ergodic Theory Dynam. Systems} \textbf{30} (2010)
441-456.

\bibitem{Yorke1}
S. M. Hammel, J. A. Yorke, and C Grebogi. Do numerical orbits of chaotic dynamical processes represent true orbits. \textit{J. of Complexity} \textbf{3} (1987), 136-145.

\bibitem{Yorke2}
S. M. Hammel, J. A. Yorke, C. Grebogi. Numerical orbits of chaotic processes represent true orbits. \textit{Bulletin of the
American Mathematical Society} \textbf{19} (1988), 465–469

\bibitem{Vietnam}  Nguyen Thieu Huy, Nguyen Van Minh. Exponential dichotomy of difference equations and applications to evolution equations on the half-line. \textit{Advances in difference equations, III.
Comput. Math. Appl.} \textbf{42} (2001), 301–311.


\bibitem{PujKor}
A. Koropecki, E. Pujals. Consequences of the Shadowing Property in low dimensions. \textit{Ergodic Theory and Dynamical Systems} Available on CJO 2013
doi:10.1017/etds.2012.195


\bibitem{Maizel}
A. D. Maizel. On stability of solutions of systems of differential equations. \textit{Trudi Uralskogo Politekhnicheskogo Instituta, Mathematics} \textbf{51} (1954), 20–50.

\bibitem{Mane2}
R. \Mane. Characterizations of AS diffeomorphisms. in:
Geometry and Topology, Lecture Notes Math., vol. 597, Springer,
Berlin, 1977, 389-394.


\bibitem{Mor}
A. Morimoto. The method of pseudo-orbit tracing and stability
of dynamical systems. \textit{Sem. Note} {\bf 39}, Tokyo Univ., 1979.

\bibitem{LipPerSh}
A. V. Osipov, S. Yu. Pilyugin, and S. B. Tikhomirov. Periodic
shadowing and $\Omega$-stability. \textit{Regul. Chaotic Dyn.} \textbf{15} (2010), 404-417.


\bibitem{PalmBook}
K. Palmer. Shadowing in Dynamical Systems. Theory and
Applications. Kluwer, Dordrecht, 2000.

\bibitem{Palm1}
K. J. Palmer. Exponential dichotomies and transversal
homoclinic points.  \textit{J. Differ. Equat.} {\bf 55} (1984), 225-256.

\bibitem{Palm2}
K. J. Palmer. Exponential dichotomies and Fredholm operators.
\textit{Proc. Amer. Math. Soc.} {\bf 104} (1988), 149-156.

\bibitem{PalmNew}
K. J. Palmer. Exponential dichotomies, the shadowing lemma and transversal homoclinic points. Dynamics reported, Vol. 1, 265-306,
Dynam. Report. Ser. Dynam. Systems Appl.,~1, Wiley, Chichester, 1988.


\bibitem{PilBook}
S. Yu. Pilyugin. Shadowing in Dynamical Systems. Lecture Notes
Math., vol. 1706, Springer, Berlin, 1999.

\bibitem{PilRev} Pilyugin S. Yu. Theory of pseudo-orbit shadowing in dynamical systems. \textit{Diff. Eqs.} \textbf{47} (2011), 1929--1938.

\bibitem{PilVar}
S. Yu. Pilyugin. Variational shadowing. \textit{Discr. Cont. Dyn. Syst., ser. B} \textbf{14} (2010) 733-737.

\bibitem{PilGen}
S. Yu. Pilyugin. Generalizations of the notion of
hyperbolicity. \textit{J. Difference Equat. Appl.} {\bf 12} (2006), 271-282.

\bibitem{PilTikhLipSh}
S. Yu. Pilyugin, S. B. Tikhomirov. Lipschitz Shadowing implies
structural stability. \textit{Nonlinearity} \textbf{23} (2010), 2509-2515

\bibitem{Pli}
V. A. Pliss. Bounded solutions of nonhomogeneous linear systems
of differential equations. Probl. Asympt. Theory Nonlin. Oscill.,
Kiev, 1977, 168-173.


\bibitem{Rob}
C. Robinson. Stability theorems and hyperbolicity in dynamical
systems. \textit{Rocky Mount. J. Math.} {\bf 7} (1977), 425-437.



\bibitem{Saw}
K. Sawada. Extended $f$-orbits are approximated by orbits.
\textit{Nagoya Math. J.} {\bf 79} (1980), 33-45.

\bibitem{Slyus}
V. Slyusarchuk.
Exponential dichotomy of solutions of discrete systems.
\textit{Ukrain. Mat. Zh.} \textbf{35} (1983), 109–115.

\bibitem{Todorov}
D. Todorov. Generalizations of analogs of theorems of Maizel and Pliss and their application in Shadowing Theory. \textit{Discr. Cont. Dyn. Syst., ser. A} \textbf{33} (2013), 4187-4205





\end{thebibliography}
\end{document}